\documentclass[preprint,12pt]{elsarticle}

\usepackage{amssymb}

\usepackage{lineno}
\usepackage{graphicx}
\usepackage{amsmath}
\usepackage{amssymb}
\usepackage{booktabs}
\usepackage{algorithm}
\usepackage{algorithmicx}
\usepackage{CJK}
\usepackage{algpseudocode}
\usepackage{ntheorem}
\newtheorem*{remark}{Remark}

\newtheorem*{proof}{Proof}
\newtheorem{corollary}{Corollary}
\newtheorem{property}{Property}
\newtheorem*{example}{Example}
\usepackage{tikz}
\usepackage{subcaption}
\usepackage{natbib} 
\newcommand*{\circled}[1]{\lower.7ex\hbox{\tikz\draw (0pt, 0pt)%
    circle (.5em) node {\makebox[1em][c]{\small #1}};}}

\definecolor{colpurple}{rgb}{0.5,0,0.8}
\definecolor{colblue}{rgb}{0, 0, 1}
\definecolor{colred}{rgb}{1, 0, 0}
\definecolor{collwh}{RGB}{129,216,207}
\definecolor{colhzy}{RGB}{0,135,69}

\AtBeginDocument{%
  \paperwidth=\dimexpr
    1in + \oddsidemargin
    + \textwidth
    + 1in + \oddsidemargin
  \relax
  \paperheight=\dimexpr
    1in + \topmargin
    + \headheight + \headsep
    + \textheight
    + 1in + \topmargin
  \relax
  \usepackage{geometry}\relax
}

\let\oldequation\equation
\let\oldendequation\endequation
\renewenvironment {equation}
{\linenomathNonumbers\oldequation}
{\oldendequation\endlinenomath}

\begin{document}

\begin{frontmatter}



\title{Lightning graph matching}


\author[XJTLU,LU,bnubj,bnuzh]{Binrui Shen}
\ead{binrui.shen@bnu.edu.cn; Binrui has moved from XJTLU to BNU.}
\author[XJTLU]{Qiang Niu}
\ead{qiang.niu@xjtlu.edu.cn}
\author[bnu,uic]{Shengxin Zhu$_\nmid$}
\ead{Shengxin.Zhu@bnu.edu.cn} 

\address[XJTLU]{Department of Applied Mathematics, School of Mathematics and Physics, Xi'an Jiaotong-Liverpool University, Suzhou 215123, P.R. China}

\address[LU]{Department of Mathematical Sciences, School of Physical Sciences, University of Liverpool, Liverpool, United Kingdom}

\address[bnubj]{School of Mathematical Sciences, Laboratory of Mathematics and Complex Systems, MOE, Beijing Normal University, 100875 Beijing, P.R.China}

\address[bnuzh]{Faculty of Arts and Sciences, Beijing Normal University, 519087 Zhuhai, P.R.China}

\address[bnu]{Research Center for Mathematics, Advanced Institute of Natural Science, Beijing Normal University, Zhuhai 519087, P.R.China}

\address[uic]{Guangdong Provincial Key Laboratory of Interdisciplinary Research and Application for Data Science, BNU-HKBU United International College, Zhuhai 519087, P.R.China}





\begin{abstract}
The spectral matching algorithm is a classic method for finding correspondences between two graphs, a fundamental task in pattern recognition. It has a time complexity of $\mathcal{O}(n^4)$ and a space complexity of $\mathcal{O}(n^4)$, where $n$ is the number of nodes. However, such complexity limits its applicability to large-scale graph matching tasks. This paper proposes a redesign of the algorithm by transforming the graph matching problem into a one-dimensional linear assignment problem. This transformation enables efficient solving by sorting two $n \times 1$ vectors. The resulting algorithm is named the Lightning Spectral Assignment Method (LiSA), which enjoys a complexity of $\mathcal{O}(n^2)$. Numerical experiments demonstrate the efficiency of LiSA, supported by theoretical analysis.
\end{abstract}




\end{frontmatter}


\section{Introduction}
\label{sec:intro}
Graph matching aims to find correspondences between nodes of two graphs, similar to point matching. The difference lies in the fact that while point matching only utilizes nodes' attributes, graph matching can additionally incorporate edges' attributes. In recent years, graph matching has been used in activity analysis \cite{chen2012efficient},
shape matching \cite{2010Learning,2011Scale}, detection of similar pictures \cite{shen2020fabricated, shen2023patent}, graph similarity computation \cite{lan2022aednet,lan2022more} and face authentication \cite{wiskott1999face, kotropoulos2000frontal}. Figure \ref{fig:face} illustrates an application in face authentication.
\begin{figure}[ht]
    \centering
    \includegraphics[width=0.5\textwidth]{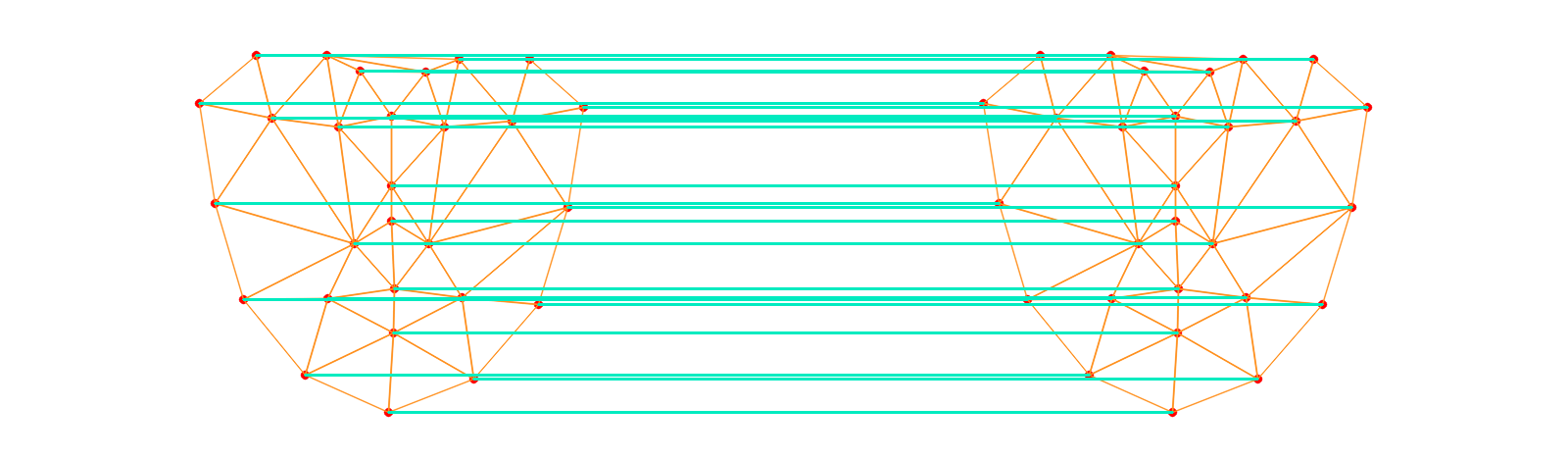}
        \includegraphics[width=0.4\textwidth]{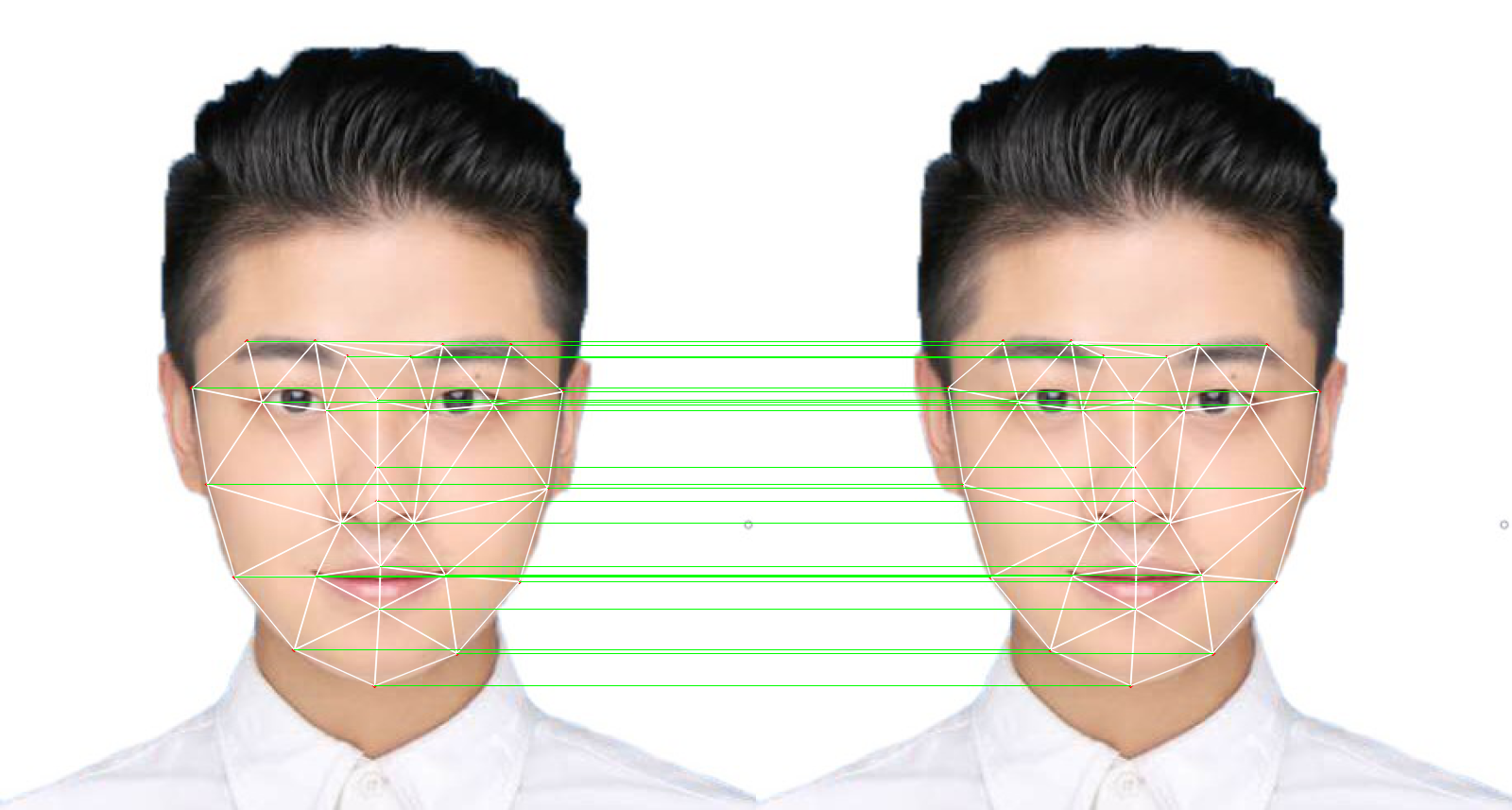}
    \caption{Face matching. In the left picture, the red points are extracted from an author's face; orange lines indicate the relations between nodes; blue lines represent correspondences between nodes of two graphs.}
    \label{fig:face}
\end{figure}

In one-to-one matching, a matching matrix $M$ can represent correspondences between two graphs:
\begin{linenomath}
\begin{equation}
M_{i j}=\left\{\begin{array}{ll}
1 & \text { if node } i \text { in graph } \mathbb{G} \text { corresponds to node } j \text { in } \widetilde{\mathbb{G}}, \\
0 & \text { otherwise. }
\end{array}\right.
\end{equation}
\end{linenomath}
$M$ is a \textit{permutation matrix} when graphs have the same cardinality of the vertices, i.e., $n = \tilde{n}$, referred to as the graph size in this paper. The set of permutation matrices is denoted by $\Pi_{n \times n} = \{M\mathbf{1}=\mathbf{1},M^{T}\mathbf{1}=\mathbf{1}, M\in\{0,1\}^{n\times n}\}.$
The information of two matching graphs can be stored in \textit{affinity matrices} $A$ and $\tilde{A}$, where $A_{ii}$ encodes the feature of $i_{th}$ node and $A_{ij}$ represents the weight of the edge between $i_{th}$ node and $j_{th}$ node.

The point matching problem is defined only by the nodes' features. The key step is to construct a matrix $U \in \mathbb{R}^{n \times n} $ such that
\begin{equation}
    U_{i j}=\Omega_v\left(A_{ii}, \tilde{A}_{jj}\right),
\end{equation}
where $\Omega_v$ is an affinity measure. The correspondences can be obtained by maximizing the total \textit{affinity score}
\begin{equation}
        M^* = \arg \max_{M \in \Pi_{n \times n}}  \ \ \sum M_{ij} U_{ij},
    \label{eq: assignment problem}
\end{equation}
where the constraints enforce the one-to-one mapping. The problem \eqref{eq: assignment problem} is also called the linear assignment problem, and $U$ is called  \textit{profit matrix}. Such an assignment problem can be solved by the Hungarian algorithm \cite{kuhn1955hungarian} with a cost $\mathcal{O}(n^3)$.

Graph matching problems further consider the similarities between edges scored by an affinity measure \(\Omega_e\). The similarities can be stored in a matrix $W \in \mathbb{R}^{n^2 \times n^2}$: 
$$W_{i \tilde{i},j \tilde{j}} = 
\begin{cases}
\Omega_v(A_{ij},\tilde{A}_{\tilde{i} \tilde{j}}) & \text { for } i= j, \tilde{i} =\tilde{j}
\\ 
\Omega_e(A_{ij},\tilde{A}_{\tilde{i} \tilde{j}}) & \text { for } i\neq j, \tilde{i} \neq \tilde{j}
\\
\ \ \ \ \ \ \ \ \ 0 &\ \text {else.}
\end{cases}
$$
Therefore, the affinity score is
\begin{equation}
    \sum M_{i\tilde{i}}M_{j\tilde{j}} W_{i \tilde{i},j \tilde{j}},
\end{equation}
which yields the following \textit{Lawler quadratic problem }\cite{lawler1963quadratic}
\begin{linenomath}
\begin{equation}
    \mathbf{m}^* = \arg \max \mathbf{m}^{T} W \mathbf{m},
    \label{eq:IQP}
\end{equation}
\end{linenomath}
where $\mathbf{m} \in \{0,1\}^{n^2 \times1 }$ and $\mathbf{m} = \operatorname{vec}(M)$. If $W \in \mathbb{R}^{n^2 \times n^2}$ is dense, most algorithms based on \eqref{eq:IQP} may suffer from time complexity $\mathcal{O}(n^4)$ and space complexity $\mathcal{O}(n^4)$ \cite{gold1996graduated, leordeanu2009integer, 2010Reweighted}. For efficiency, \citet{zaslavskiy2008path} consider a case that $W=A \otimes \tilde{A}$, the problem  \eqref{eq:IQP} is equivalent to 
\begin{equation}
M^* = \arg \min_{M \in \Pi_{n \times n}} \ \ \left\|A-M\widetilde{A}M^{T}\right\|_{F}^{2}.
\label{eq:KBQP}
\end{equation}
which is called the \textit{Koopmans-Beckmann quadratic problem} \cite{1957Assignment}. The core of \eqref{eq:KBQP}  is to minimize the discrepancies between two graphs (detailed derivation details of \eqref{eq:IQP} to \eqref{eq:KBQP} can be found in \cite{lu2016fast}). Under this formulation, computation only involves a few $n \times n$ matrices rather than an $n^2 \times n^2$ one. 

Since graph matching is generally an NP-hard problem \cite[GT49]{garey1979computers}, many recent publications on graph matching focus on continuous relaxation, aiming to find sub-optimal solutions at an acceptable cost. Common approaches include, but are not limited to, bipartite matching \cite{serratosa2014fast}, spectral relaxation \cite{2010Reweighted, cour2006balanced},
doubly stochastic relaxation \cite{leordeanu2009integer,lu2016fast,shen2022dyna, shen2024adaptive},
continuous optimization \cite{zaslavskiy2008path} and probabilistic modeling \cite{egozi2012probabilistic}. Among recently proposed graph matching algorithms, spectral-based methods \cite{umeyama1988eigendecomposition, luo2003spectral, caelli2004inexact,leordeanu2005spectral,robles2007riemannian} have received a lot of attention due to their efficiency. These approaches rely on the fact that eigenvalues and the eigenvectors of a graph's adjacency or Laplacian matrix are invariant to node permutation.

The computational complexity of algorithms utilizing continuous relaxation is generally no less than $\mathcal{O}(n^3)$. The essence of relaxation lies in solving the original problem under relaxed constraints initially, followed by converting the continuous solution back to the discrete domain. The second step involves solving a linear assignment problem where the profit matrix represents a continuous solution. This perspective suggests that graph matching and point matching share similar processes: constructing a profit matrix and resolving the underlying assignment problem. A fast doubly stochastic projected fixed-point method (DSPFP) \cite{lu2016fast} requires at least $\mathcal{O}(n^3)$ to obtain a continuous solution, and the general assignment algorithms have a worst-case complexity of $\mathcal{O}(n^3)$ as well. In this framework, graph matching algorithms based on relaxation have at least $\mathcal{O}(n^3)$ cost. Such complexity is expensive for large graph matching problems.

In this work, we contribute a novel graph matching algorithm LiSA with only $\mathcal{O}(n^2)$ time complexity by redesigning the classic spectral matching (SM) \cite{leordeanu2005spectral} with worst-case complexity of $\mathcal{O}(n^4)$. Such lightning speed comes from transforming the graph matching problem into a one-dimensional linear assignment problem. To begin with, we aggregate two graphs into two sets of points using spectral scores derived from the principal eigenvectors. These points are matched according to the scores: the node with the largest spectral score in $\mathbb{G}$ matches the node with the largest spectral score in $\tilde{\mathbb{G}}$; the second largest one matches another second largest one and so on. A more general method is Umeyama's algorithm \cite{umeyama1988eigendecomposition} that uses all eigenvectors. We prove that LiSA can find the perfect correspondences when matching graphs are isomorphic. A robustness analysis of LiSA is also given.

The rest of the paper is organized as follows. In Section 2, we first review SM and introduce its fast implementation for the Koopmans-Beckmann quadratic problem. We introduce LiSA in section 3. The invariant properties and robustness of LiSA are studied in section 4. In the last section, we give the concluding remarks and discuss ideas for future works.


        

\section{Spectral matching}
In this section, we recall the SM \cite{leordeanu2005spectral} and derive its fast implementation for the Koopmans-Beckmann quadratic problem.

\subsection{Spectral Matching Algorithm}
To relax problem \eqref{eq:IQP}, \citet{leordeanu2005spectral} drop the integer constraints, then a soft solution can be obtained from
\begin{equation}
    \mathbf{x}^* = \arg \max _\mathbf{x}\frac{\mathbf{x}^T W \mathbf{x}}{\mathbf{x}^{T} \mathbf{x}}.
    \label{eq:RelaxIQP}
\end{equation}
 They interpret the $X_{ij}^{*} \in [0, 1]$, $\mathbf{x}^* = \operatorname{vec}(X^*)$, as the association of the corresponding matching pair between node $i$ of $\mathbb{G}$ and node $j$ of $\tilde{\mathbb{G}}$. The Rayleigh's ratio theorem \cite[Sec. 4.1]{hilbert1985methods} states that $\mathbf{x}^*$ is the principal eigenvector of $W$, which can be numerically approximated by the power method:
\begin{equation}
    \mathbf{x}^{(k)} =  \frac{W\mathbf{x}^{(k-1)}}{\max(W\mathbf{x}^{(k-1)})}.
    \label{eq:spectral}
\end{equation}
Since only the relative values between the elements matter, the maximum of $\mathbf{x}$ is scaled to 1. Each iteration needs $\mathcal{O}(n^4)$ operations since $W \in \mathbb{R}^{n^2 \times n^2}$. The algorithm of the power method is shown in Algorithm \ref{ag:SM}. If $W$ is sparse, variants of the Lanczos method can compute its principle eigenvector with a cost less than $\mathcal{O}(n^3)$  \cite{leordeanu2005spectral}.

The solution is converted back into the original discrete domain by solving the assignment problem \eqref{eq: assignment problem} defined by the soft solution $\mathbf{x}^*$. \citet{leordeanu2005spectral} propose a greedy algorithm, which is slightly faster than the Hungarian method in practice \cite{lu2016fast}. However, the greedy algorithm does not guarantee the optimal solution.  

\begin{CJK*}{UTF8}{gkai}
    \begin{algorithm}
        \caption{The power method}
        \begin{algorithmic}[1] 
            \Require $W$
            \State $Initial$  $\mathbf{x}^{(0)} = \mathbf{1}$
                \For{$k=1,2 \dots, $ until convergence}
                    \State $\mathbf{z}^{(k)} = W\mathbf{x}^{(k-1)}$
                    \State $\mathbf{x}^{(k)} = \frac{\mathbf{z}^{(k)}}{\max(\mathbf{z}^{(k)})}$
                \EndFor
                \State \Return{$\mathbf{x}$}
        \end{algorithmic}
        \label{ag:SM}
    \end{algorithm}
\end{CJK*}

\subsection{Spectral matching for the Koopmans-Beckmann quadratic problem}

For the Koopmans-Beckmann quadratic problem \eqref{eq:KBQP}, we can rewrite the iterative formula \eqref{eq:spectral} in matrix form
\begin{equation}
    X^{(k)} = \frac{AX^{(k-1)}\tilde{A}}{\max(AX^{(k-1)}\tilde{A})},
\end{equation}
which use the fact that $ (A\otimes \tilde{A}) \mathbf{x} = \operatorname{vec}{(AX\tilde{A})}$ \cite[Thm.16.2.1]{Harville2008Matrix}.  With this formulation, the computational cost of each iteration is reduced to $\mathcal{O}(n^3)$ \cite{lu2016fast}.  The Spectral matching for the Koopmans-Beckmann quadratic problem (SM-KB) is displayed in Algorithm \ref{ag.SMKB}.

\begin{CJK*}{UTF8}{gkai}
    \begin{algorithm}
        \caption{Spectral matching for the Koopmans-Beckmann quadratic problem (SM-KB)}
        \begin{algorithmic}[1] 
            \Require $A,\tilde{A}$
            \Ensure $M$       
                 \State $Initial$  $X = \mathbf{1 1}_{n \times \tilde{n}}^{T} / n \tilde{n}$
                 \For{$k=1,2 \dots, $ until convergence}
                    \State $Z^{(k)} =AX^{(k-1)}\tilde{A}$
                    \State $X^{(k)}=\frac{Z^{(k)}} {\max(Z^{(k)})} $
                \EndFor
                \State Discretize X to obtain $M$
                \State \Return{$M$}
        \end{algorithmic}
        \label{ag.SMKB}
    \end{algorithm}
\end{CJK*}

\section{Lightning Spectral Assignment method}
In this section, we introduce the lightning spectral assignment method for graph matching tasks formulated by the Koopmans-Beckmann quadratic problem.
\subsection{Lightning eigenvector method}
We introduce a method to efficiently compute the leading eigenvector of $W = A \otimes \tilde{A}$. The problem \eqref{eq:RelaxIQP} can be reformulated as
\begin{equation}
    \mathbf{x}^* = \arg \max_\mathbf{x}\frac{\mathbf{x}^{T} (A \otimes \tilde{A}) \mathbf{x}}{\mathbf{x}^{T} \mathbf{x}}.
    \label{eq:relaxKB}
\end{equation}
Notice that $\mathbf{x^*} = \varphi \otimes \tilde{ \varphi}$ \cite[Thm.21.11.1]{Harville2008Matrix}, where $\varphi$ and $\tilde{\varphi}$ are leading eigenvectors of $A$ and $\tilde{A}$, respectively. Replacing $\mathbf{x}$ by $\mathbf{v} \otimes \tilde{\mathbf{v}}$ where $\mathbf{v}, \tilde{\mathbf{v}} \in \mathbb{R}^{n \times 1}$, we have 

\begin{equation}
\frac{\mathbf{x}^T W \mathbf{x}}{\mathbf{x}^{T} \mathbf{x}}=\frac{(\mathbf{v} \otimes \tilde{\mathbf{v}})^{T}(A \otimes \tilde{A})(\mathbf{v} \otimes \tilde{\mathbf{v}})}{\left(\mathbf{v} \otimes \tilde{\mathbf{v}}\right)^{T}(\mathbf{v} \otimes \tilde{\mathbf{v}})} 
\end{equation}
According to \cite[Lem. 16.1.2]{Harville2008Matrix},
\begin{equation}
\begin{aligned}
\frac{(\mathbf{v} \otimes \tilde{\mathbf{v}})^{T}(A \otimes \tilde{A})(\mathbf{v} \otimes \tilde{\mathbf{v}})}{(\mathbf{v} \otimes \tilde{\mathbf{v}})^{T}(\mathbf{v} \otimes \tilde{\mathbf{v}})} =&\frac{(\mathbf{v}^{T} A \mathbf{v}) \times(\tilde{\mathbf{v}}^{T} \tilde{A} \tilde{\mathbf{v}})}{\left(\mathbf{v}^{T} \mathbf{v}\right) \times\left(\tilde{\mathbf{v}}^{T} \tilde{\mathbf{v}}\right)} 
\\
=&\frac{\mathbf{v}^{T} A\mathbf{v}}{\mathbf{v}^{T} \mathbf{v}} \times \frac{\tilde{\mathbf{v}}^{T} \tilde{A} \tilde{\mathbf{v}}}{\tilde{\mathbf{v}}^{T} \tilde{\mathbf{v}}}.\end{aligned}
\end{equation}
We can rewrite the problem \eqref{eq:relaxKB} as two sub-problems:
\begin{equation}
    \begin{aligned}
    \mathbf{x}^* & = \varphi \otimes \tilde{\varphi}, \\
    \varphi &= \arg \max _\mathbf{v}\frac{\mathbf{v}^{T} A \mathbf{v}}{\mathbf{v}^{T} \mathbf{v} }, \\
    \tilde{\varphi}&= \arg \max_{\tilde{\mathbf{v}}}\frac{\tilde{\mathbf{v}}^{T} \tilde{A} \tilde{\mathbf{v}}}{\tilde{\mathbf{v}}^{T} \tilde{\mathbf{v}}}. 
    \label{eq:lightning relaxedKB}
    \end{aligned}
\end{equation}
The matrix form of $\mathbf{x}^*$ can be obtained by
\begin{equation}
    X^* = \mathbf{\varphi} \otimes \tilde{\mathbf{\varphi}}^T,
    \label{eq:1D_profit}
\end{equation}
where $\mathbf{x}^* = \operatorname{vec}(X^*)$.

\subsection{One-dimensional assignment}


 
The matching matrix is obtained from the assignment problem based on the profit matrix $X^*$:
\begin{align}
      M^* =& \arg \max_{M \in \Pi_{n \times n}} \sum M_{ij}X^*_{ij} \\
=& \arg \max_{M \in \Pi_{n \times n}} \sum M_{ij}\mathbf{\varphi}_i\tilde{\mathbf{\varphi}}_j\\
=& \arg \max_{M \in \Pi_{n \times n}} \sum  \mathbf{\varphi}_i\tilde{\mathbf{\varphi}}_{M(i)}.
\label{eq:1D_permutaion}
  \end{align}
where $\tilde{\mathbf{\varphi}}_{M(i)}$ represents  $\tilde{\mathbf{\varphi}}_{i}$ permuted by $M$ i.e., $M_{iM(i)} =1$. This is a one-dimensional assignment problem that can be computed by sorting with $\mathcal{O}(n \ln n)$ operations \cite{rabin2012wasserstein} according to the rearrangement inequality. 
\begin{theorem*}

Rearrangement inequality \cite{hardy1952inequalities} 
\\ Let
$$
z_1 \geq \cdots \geq z_n \quad \text { and } \quad y_1 \geq \cdots \geq y_n,
$$
then
$$
z_{\sigma(1)} y_1+\cdots+z_{\sigma(n)} y_n \leq z_1 y_1+\cdots+z_n y_n,
$$
where 
$z_{\sigma(1)}, \ldots, z_{\sigma(n)}$ represent permutations of $z_1, \ldots, z_n$.
\end{theorem*} 
The rearrangement inequality theorem indicates that the sum achieves maximum when two series follow the same order. It suggests that we can solve such an assignment problem by sorting $\mathbf{\varphi}$ and $\tilde{\mathbf{\varphi}}$ in descending order. To make the idea more straightforward to understand, we give an example to illustrate the process:
\begin{example} \textbf{A one-dimensional assignment}

For given two vectors $\mathbf{\varphi}$ and $\tilde{\mathbf{\varphi}}$
 $$
[\mathbf{\varphi}|index]=\left[\begin{array}{l|l}
\ 4\  &{\circled{1}}\\
\ 5\  & {\circled{2}}\\
\ 6\  & {\circled{3}}\end{array}\right], [\widetilde{\mathbf{\varphi}}|index]=\left[\begin{array}{c|c}
\ 4\  & {\circled{1}}\\
\ 6\  & {\circled{2}}\\
\ 5\  & {\circled{3}}\end{array}\right],
$$
where numbers in circles represent the original indexes. By sorting, we have  
$$
\operatorname{sort}([\mathbf{\varphi}|index])=\left[\begin{array}{l|l}
\ 6\  & {\circled{3}}\\
\ 5\  & {\circled{2}}\\
\ 4\  & {\circled{1}}\end{array}\right], \operatorname{sort}(\widetilde{\mathbf{\varphi}}|index)=\left[\begin{array}{c|c}
\ 6\  & {\circled{2}}\\
\ 5\  & {\circled{3}}\\
\ 4\  & {\circled{1}}\end{array}\right].
$$
The two vectors can be matched directly,
 $$\left[\begin{array}{l}
( {\circled{3}}, {\circled{2}}) \\
( {\circled{2}}, {\circled{3}}) \\
( {\circled{1}}, {\circled{1}})
\end{array}\right] \Leftrightarrow  \ 
M =\left[ \begin{matrix}  
  1,\ 0,\ 0
\\0,\ 0,\ 1
\\0,\ 1,\ 0
 \end{matrix} \right].$$
\end{example}

The one-dimensional assignment also applies to the unbalanced case that $n < \tilde{n}$. We generalize the rearrange theorem to the following corollary.
\begin{corollary} Rearrangement inequality for the unbalanced case ($n < \tilde{n}$)
\\ Let
$$
z_1 \geq \cdots \geq z_n \geq \cdots \geq z_{\tilde{n}} \quad \text { and } \quad y_1 \geq \cdots \geq y_n,
$$
then
$$
z_{\sigma(1)} y_1+\cdots+z_{\sigma(n)} y_n \leq z_1 y_1+\cdots+z_n y_n,
$$
where $z_{\sigma(1)}, \ldots, z_{\sigma(n)}$ represent permutations of $z_1, \ldots, z_n$.
\end{corollary}
\begin{proof}
Let $y_{n+1} = y_{n+2} \ldots = y_{\tilde{n}} =0$. According to the rearrangement inequality,
$$
z_{\sigma(1)} y_1+\cdots+z_{\sigma({\tilde{n}})} y_{\tilde{n}} \leq z_1 y_1+\cdots+z_{\tilde{n}} y_{\tilde{n}}.
$$
By dropping zero terms, we have 
$$
z_{\sigma(1)} y_1+\cdots+z_{\sigma(n)} y_{n} \leq z_1 y_1+\cdots+z_n y_n,
$$
which completes the proof.
\end{proof}
The corollary allows us to solve the unbalanced assignment problem similarly: by matching sorted $\mathbf{\varphi}$ with the top $n$ elements of sorted $\tilde{\mathbf{\varphi}}$ directly.

We give a classic case to illustrate the idea of this assignment algorithm: assigning $n$ workers to complete $n$ tasks alone. Assume $\tilde{\mathbf{\varphi}}_i$ represents the capability of the person $i$ and $\mathbf{\varphi}_i$ represents the difficulty of the task $i$. We assign the best worker to complete the most challenging task, followed by the second best worker to complete the second most challenging task; the remaining tasks can be assigned in the same order. This idea is reminiscent of a famous line from the Spider-Man movie: "With great power comes great responsibility."

In another view, the assignment problem \eqref{eq:1D_permutaion} is equivalent to

\begin{equation}
    M^* = \arg \min_{M \in \Pi_{n \times n}} \|\mathbf{\varphi}-M\tilde{\mathbf{\varphi}}\|^2_2,
    \label{eq:assignment difference}
\end{equation}
with 
\begin{equation}
    \|\mathbf{\varphi}-M\tilde{\mathbf{\varphi}}\|^2_2 = \sum_i( \mathbf{\varphi}_i\tilde{\mathbf{\varphi}}_{M(i)}+ \underbrace{\mathbf{\varphi}_i^2 +\tilde{\mathbf{\varphi}}_i^2}_{constants}).
\end{equation}
From the problem \eqref{eq:assignment difference}, we can find an assignment by minimizing the disagreements.
\subsection{Algorithm}
We obtain a lightning spectral assignment matching algorithm (LiSA) by combining the lightning eigenvector method and one-dimensional assignment. The main idea is based on the fact that the eigenvectors of the affinity matrix of a graph are invariant concerning node permutation. The whole algorithm is shown in Algorithm \ref{ag.LiSA}. Regardless of sparse matrix computation, LiSA has time complexity only $\mathcal{O}(n^2)$ and space complexity $\mathcal{O}(n^2)$. The power method with $\mathcal{O}(n^2)$ computational cost is used for step 1 and step 2. The one-dimensional assignment only has $\mathcal{O}(n \ln n)$ time cost. 
\begin{CJK*}{UTF8}{gkai}
    \begin{algorithm}
        \caption{Lightning spectral assignment method (LiSA)}
        \begin{algorithmic}[1] 
            \Require $A,\tilde{A}$
            \Ensure $M$       
                \State  $\varphi = \arg \max \limits_\mathbf{v}\frac{\mathbf{v}^{T} A \mathbf{v}}{\mathbf{v}^{T} \mathbf{v} }$ 
                \State $\tilde{\varphi}= \arg \max\limits_{\tilde{\mathbf{v}}}\frac{\tilde{\mathbf{v}}^{T} \tilde{A} \tilde{\mathbf{v}}}{\tilde{\mathbf{v}}^{T} \tilde{\mathbf{v}}}$
                \State One-dimensional assignment $(\mathbf{\varphi},\tilde{\mathbf{\varphi}}) \ \Rightarrow M$
                \State \Return{$M$}
        \end{algorithmic}
        \label{ag.LiSA}
    \end{algorithm}
\end{CJK*}



Despite being a fast member of the spectral-based algorithms, LiSA can be viewed from another perspective: the algorithm transforms graphs into points with aggregated features encompassing local and structural information. This enables the matching of two sets of points by sorting their respective aggregating features. The primary process is shown in Figure \ref{fig:gap}.
\begin{figure}
    \centering
    \includegraphics[width=0.9\textwidth]{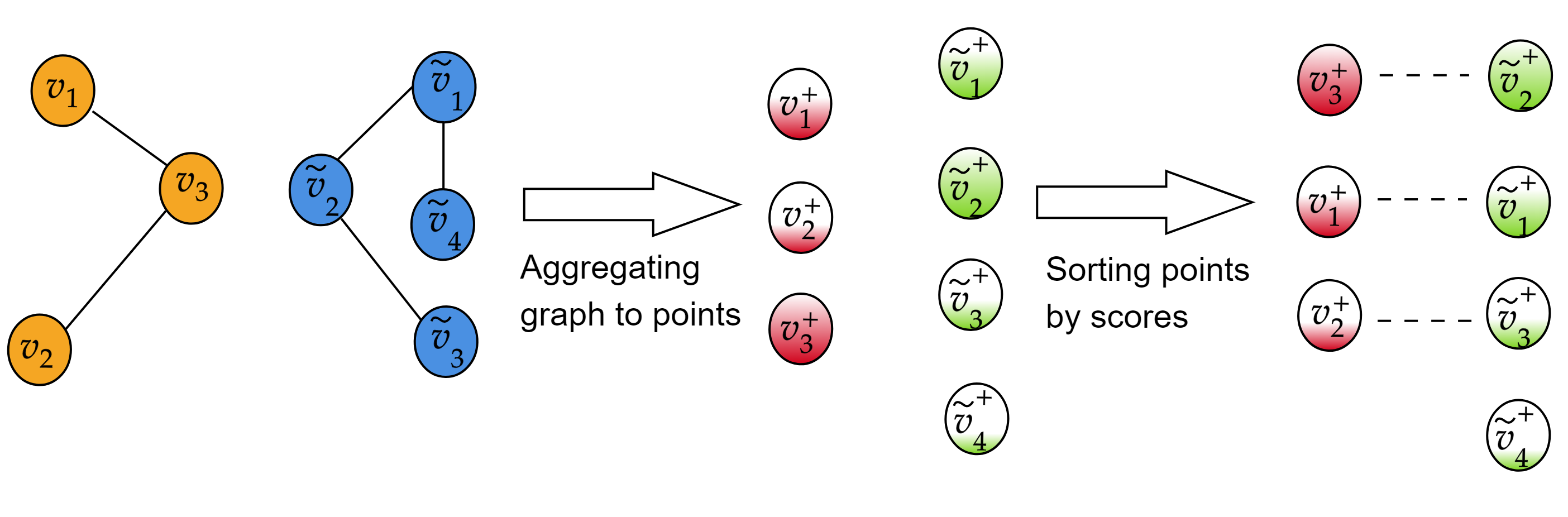}
    \caption{The graph to aggregation pointing. $v_1^+$ represents the aggregation point corresponding to $v_1$. Other notations of aggregation points follow in this manner. Richer colors of aggregation points mean larger features.}
    \label{fig:gap}
\end{figure}
In LiSA, the aggregating features of a graph are the elements of the leading eigenvector of the corresponding affinity matrix.

\section{Analysis}
In this section, we explore the LiSA algorithm in two scenarios: isomorphic graphs and graphs with random noises.

\subsection{Perfect matching for isomorphic graphs}
We demonstrate that LiSA can find the perfect matching for isomorphic graphs by leading eigenvectors.
\begin{theorem*} Isomorphic graphs are isospectral \cite[Lem. 8.1.1]{godsil2001algebraic}

Two graphs are isomorphic if and only if $A = P\tilde{A}P^T$ where $P$ is a permutation matrix. If $\mathbf{\varphi}$ and $\tilde{\mathbf{\varphi}}$ are eigenvectors corresponding to the same eigenvalue of $A$ and $\tilde{A}$,  then $\tilde{\mathbf{\varphi}} = P\mathbf{\varphi}$.
\end{theorem*}

\begin{remark}
    The converse of the above theorem is also true, i.e., if $A$ and $\tilde{A}$ are both real symmetric matrices with the same eigenvalues and there exists a permutation matrix $P$ enables 
\begin{equation}
    \tilde{\mathbf{\varphi}}_i=P\mathbf{\varphi}_i
\end{equation}
where $(\lambda_i,\mathbf{\varphi}_i)$ are eigenpair of $A$ and $(\lambda_i,\tilde{\mathbf{\varphi}}_i)$
are eigenpair of $\tilde{A}$. The permutation matrix $P$ makes 
\begin{equation}
    A = P\tilde{A}P^T,
\end{equation}
i.e., $A$ and $\tilde{A}$ are isomorphic.
\end{remark}

\begin{property}
For two isomorphic graphs $A = P\tilde{A}P^T$ where $P$ is a permutation matrix, LiSA can find the perfect matching matrix $P$ exactly.
\end{property}

\begin{proof}
According to the theorem,  $\mathbf{\varphi} = P\tilde{\mathbf{\varphi}}$ where $\mathbf{\varphi}$ and $\tilde{\mathbf{\varphi}}$ are principle eigenvectors of $A$ and $\tilde{A}$. We can obtain $P$ using the one-dimensional assignment algorithm to solve the problem \eqref{eq:1D_permutaion}. 
\end{proof}

Let $\mathbf{\phi}$ and $\tilde{\mathbf{\phi}}$ be eigenvectors corresponding to the same eigenvalue of $A$ and $\tilde{A}$. For any pair of $\mathbf{\phi}$ and $\tilde{\mathbf{\phi}}$, there exists 
\begin{equation}
    P = \arg \max_{M \in \Pi_{n \times n}} \sum  \mathbf{\phi}_i\tilde{\mathbf{\phi}}_{M(i)}, 
    \label{eq:same solution}
\end{equation}
where $\tilde{\mathbf{\phi}}_{M(i)}$ represents $\tilde{\mathbf{\phi}}_{i}$ permuted by $M$. 
This yields an absorbing counter-intuitive remark.
\begin{remark}
For $A = P\tilde{A}P^T$ and $W = A \otimes \tilde{A}$ where $P$ is a permutation matrix, the optimal solution of graph matching problem \eqref{eq:IQP} can be obtained through any pair of eigenvectors $\mathbf{\phi}$ and $\tilde{\mathbf{\phi}}$ corresponding to the same eigenvalue.
\end{remark}

\subsection{Stability analysis}
In this subsection, we discuss the stability of LiSA under deformations: $\tilde{A}=A + E$ where $E$ represents perturbations. This is a key area in the network problems \cite{2001Link,  wu2007power,wu2013accelerating, shen2015multiple}. \citet{2001Link} demonstrate that the leading eigenvector of a symmetric matrix $A$ is robust to small perturbations when \textit{eigengap} $\rho$, the difference between the first two largest eigenvalues of $A$, is large.

According to \cite[Thm.$V .2 .8$]{stewart1990matrix} and \cite[Thm.1]{2001Link}, if $\sqrt{2}\|E\|_F \leq \frac{\rho}{2}$, the perturbation of the leading eigenvector $\mathbf{\varphi}$ satisfies the inequality:

\begin{equation}
    \left\|\mathbf{\varphi}-\tilde{\mathbf{\varphi}}\right\|_2 \leq \frac{4\|E\|_{F}}{\rho-\sqrt{2}\|E\|_{F}}.
    \label{eq:perturbation}
\end{equation}
This analysis supports that small perturbation $\|E\|_{F}$ relative to $\rho$ will not significantly change the direction of the principal eigenvector. 
If we substitute $\sqrt{2}\|E\|_{F} \leq \frac{\rho}{2}$ in \eqref{eq:perturbation}, a more intuitive formula is 
\begin{equation}
    \left\|\mathbf{\varphi}-\tilde{\mathbf{\varphi}}\right\|_2 \leq 2\sqrt{2}.
\end{equation}







\section{Experiments}
We apply the proposed algorithm LiSA\footnote{The code of LiSA is available at https://github.com/BinruiShen/Lightning-graph-matching.} to the graph matching task and evaluate it from the following aspects:
\begin{itemize}
    \item Q1. Compared to other spectral-based methods, what advancements does LiSA offer ?
    \item Q2. How efficient and scalable is LiSA?
    \item Q3. When graphs have tens of millions of edges, how stable is LiSA?
\end{itemize}

\subsection{Experimental Setup}
\textbf{Datasets}: Algorithms are evaluated by three different kinds of graphs: weighted fully connected graphs, weighted sparse graphs, and unweighted sparse graphs. For a weighted graph, the weight of an edge $A_{ij}$ represents the Euclidean distance between two corresponding nodes in this paper. For unweighted graphs, $A_{ij} \in  \{0,1\}$ only indicates whether the corresponding nodes are linked or not. Graphs are from:
\begin{itemize}
    \item  Random graphs are established by generating 
$n$ pairs of random numbers to serve as the coordinates of the nodes. The \textit{Delaunay triangulation} connects the points in sparse cases.  
    \item The yeast's protein-protein interaction (PPI) \cite{vijayan2015magna++} networks contain 1,004 proteins and 4,920 high-confidence interactions.
    \item The social network comprising 'circles' (or 'friends lists') from Facebook \cite{snapnets} contains 4039 users (nodes) and 88,234 relations (edges).
    \item The social network from Twitter \cite{snapnets} includes 5,120 users (nodes) and 130,575 connections (edges).
\end{itemize}


\noindent \textbf{Baseline methods:} 
\begin{itemize}
    \item SM \cite{leordeanu2005spectral} is a classic spectral-based method to solve graph matching problems (introduced in Section 2.1). 
    \item SM-KB is a fast implementation of SM for  Koopmans-Beckmann quadratic problem (introduced in Section 2.2).
    \item DSPFP \cite{lu2016fast} is a fast doubly stochastic projected fixed-point method that has a similar iterative formula with SM-KB.
\end{itemize}


\textbf{Setting of algorithms.} LiSA and SM both use the power method to calculate the leading eigenvector. The iteration process is terminated when $\max(|\mathbf{x}^{(k)}- \mathbf{x}^{(k-1)}|)<1e^{-4}$. Since the final matching matrix depends only on the relative values between the elements of the eigenvectors, the precision of the power method is not critical. When graphs are randomly generated, all algorithms run on the same problem sets over 20 trials for each scenario. We complete all experiments in Python 3 with a single thread in an i7 2.80 GHz PC.



\subsection{Comparasion between LiSA and other spectral-based methods}
In the first set of experiments, spectral-based methods are valued by random isomorphic graph matching. The number of nodes ranges from 30 to 150 because SM runs out of memory if the number of nodes exceeds 150. Figure \ref{fig:small} illustrates that LiSA has only $\mathcal{O}(n^2)$ time cost and is significantly faster than other spectral-based algorithms.  In terms of accuracy, three algorithms achieve perfect matching in different types of graphs.



\begin{figure}[ht]
    \centering
    \subcaptionbox{Dense weighted graphs\label{fig:full}}{
        \includegraphics[width=0.25\textwidth]{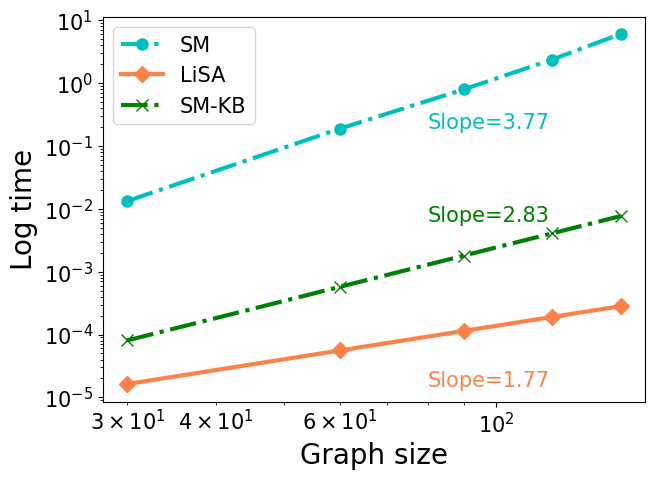}
    }
    \subcaptionbox{Sparse weighted graphs\label{fig:without_background}}{
        \includegraphics[width=0.25\textwidth]{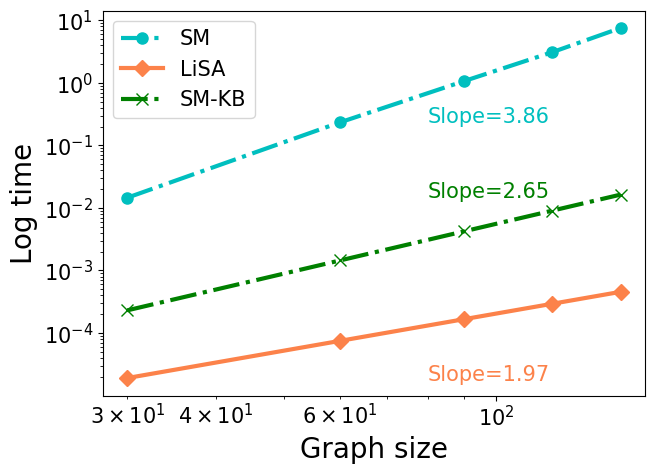}
    }
    \subcaptionbox{Sparse unweighted graphs\label{fig:without_borders}}{
        \includegraphics[width=0.25\textwidth]{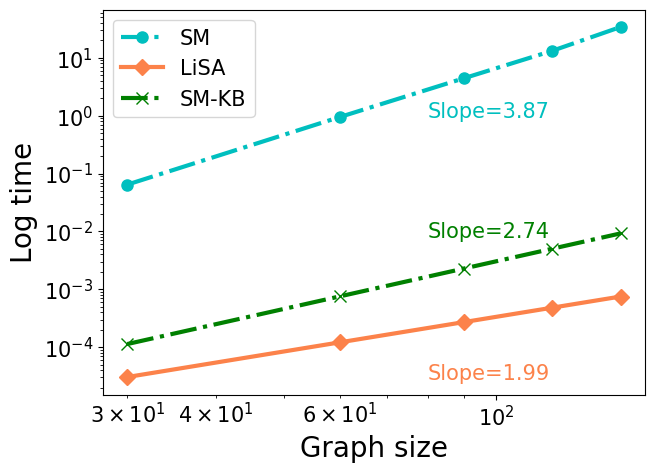}
    }
    \caption{Comparison between spectral-based algorithms.}
    \label{fig:small}
\end{figure}

\subsection{Efectiveness results}
In this subsection, we evaluate algorithms using isomorphic random graphs and graphs from real world. Figure \ref{fig:large_graphs} and Figure \ref{fig:real} show that only LiSA achieves perfect matching in all sets of graphs. In contrast, DSPFP and SM-KB only obtain optimal solutions in one set of experiments. With an increase in the number of nodes, there is a tendency for the errors of the two algorithms to potentially escalate in Figure \ref{fig:large_graphs}. Theoretically, SM-KB is capable of perfectly matching two isomorphic graphs. The failure may be attributed to a truncation error. Table \ref{tab:large} demonstrates the efficiency of LiSA: 552x and 1422x speedup with respect to SM-KB and DSPFP in fully weighted graph sets.

\begin{figure}[h]
	\centering

		\includegraphics[width=1\textwidth]{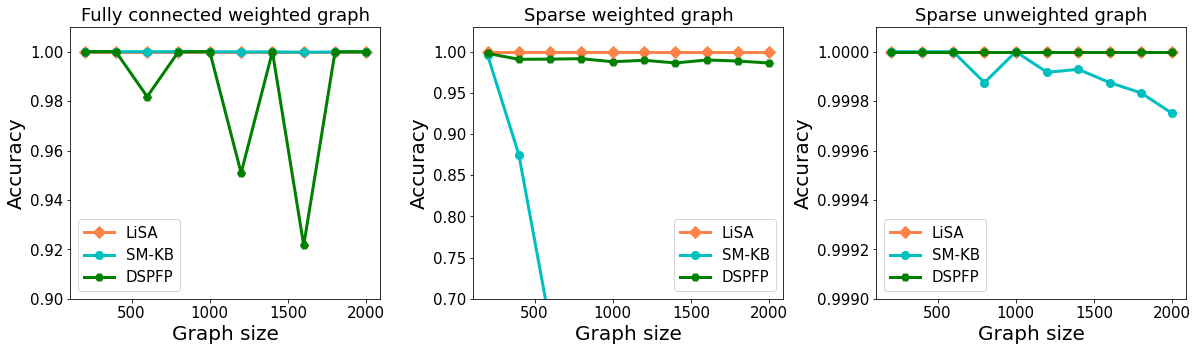}



	   \caption{Random large graphs.}
	   \label{fig:large_graphs}
\end{figure}

\begin{table}[]
\centering

\resizebox{0.8\columnwidth}{!}{%
\begin{tabular}{cccc}
\hline
Experiment type & SM-KB/LiSA & DSPFP/LiSA & Time of LiSA \\ \hline
Fully weighted   & 552x        & 1422x       & 0.02s        \\
Sparse weighted & 155x        & 680x        & 0.05s        \\
Sparse binary   & 33x         & 99x         & 0.6s         \\ \hline
\end{tabular}%
}
\caption{Running time comparison when graph size is 2000.}
\label{tab:large}
\end{table}

\begin{figure}[ht]
    \centering
    \subcaptionbox{Accuracy}{
        \includegraphics[width=0.4\textwidth]{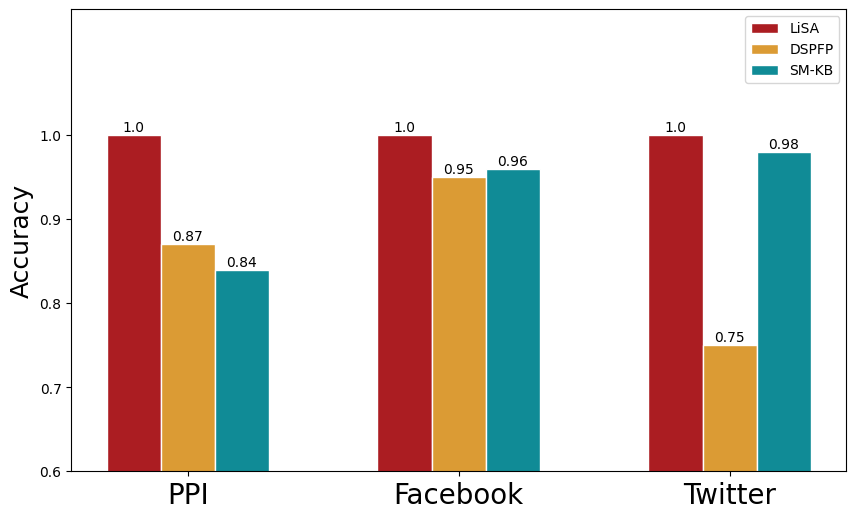}
    }\ \ \  \ \ \ 
    \subcaptionbox{Running time}{
        \includegraphics[width=0.4\textwidth]{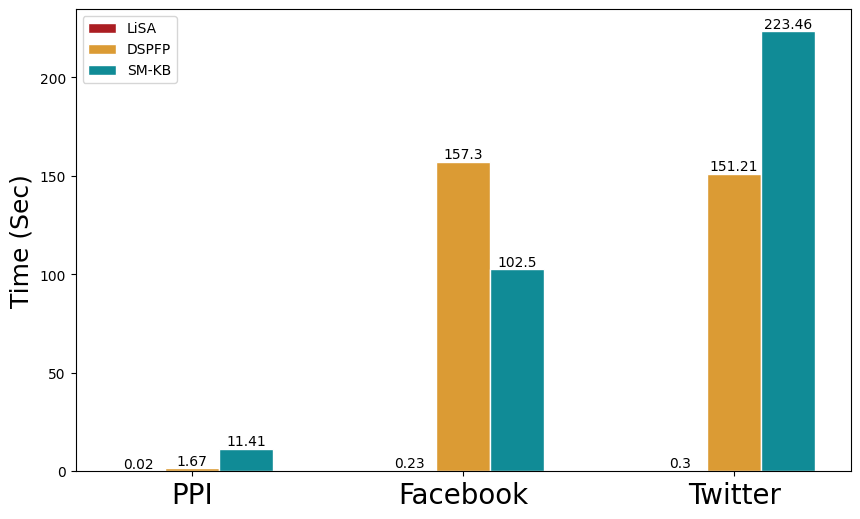}
    }
    \caption{Graphs from real world.}
    \label{fig:real}
\end{figure}

\subsection{Robustness results}
In this set of experiments, we assess the robustness of LiSA by performing fully connected weighted random graph matching under different noise levels. We construct $A$ randomly and assign random noise to random $n$ edges to construct $\tilde{A}$:
\begin{equation}
    \tilde{A}_{ij} =|A_{ij} + l\epsilon|,
\end{equation}
where $l$ represents the level of noise and $\epsilon$ is a random value ranges from $(-1\%, 1\%)$. 
\begin{figure}[h]
	\centering
		\includegraphics[width=0.9\textwidth]{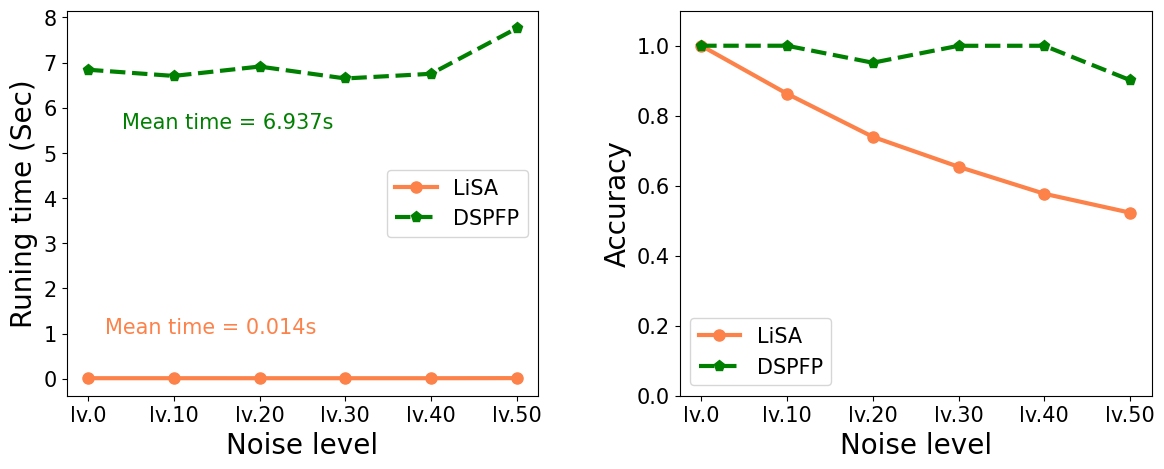}
	   \caption{Random graphs with noise.}
	   \label{fig:noi_compare}
\end{figure}
Figure \ref{fig:noi_compare} depicts that DSPFP is more robust to noise, while LiSA is much faster than it.

\subsection{Evaluation of stability}
In this subsection, we evaluate LiSA on four kinds of large graphs:  weighted fully connected graphs, weighted sparse graphs, unweighted sparse graphs, and weighted fully connected graphs with noise of level 20. The number of nodes varies from one thousand to ten thousand and the maximum number of edges reaches about fifty million.
\begin{figure}[h]
    \centering
    \includegraphics[width=0.9\textwidth]{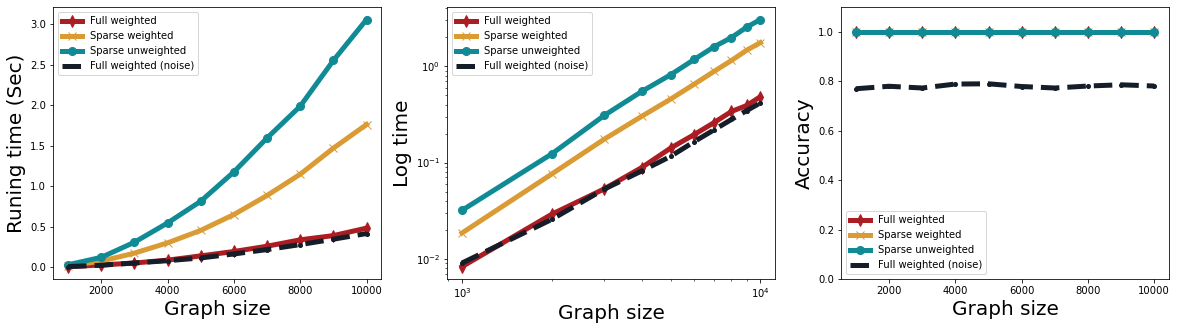}
    \caption{Matching on super large graphs.}
    \label{fig:large}
\end{figure}
As Figure \ref{fig:large} shows, the running time of LiSA is at most 3s for different types of large graphs with 10000 nodes. Its accuracy remains high and stable.

\section{Conclusions}

We propose a lightning graph matching algorithm that leverages the invariance of eigenvectors of a graph's adjacency matrix to node permutation. The finding correspondence between nodes becomes matching elements of two leading eigenvectors that can be completed efficiently by sorting. We prove that LiSA can find the perfect correspondence between two isomorphic graphs. The experimental results showcase the algorithm's robustness and scalability with respect to graph size.

A limitation of our algorithm is its inability to handle vector attributes of nodes and edges. Additionally, compared to DSPFP, LiSA is more sensitive to noise. Consequently, LiSA excels in addressing large problems with low noise. Future research may investigate machine learning techniques for aggregating graphs into points and substituting the spectral-based method to improve its robustness. Moreover, extending our algorithm to adopt high-order features, such as the angle between three nodes, is another potential avenue for exploration.

\section{Acknowledgements}

The research is supported by the Interdisciplinary Intelligence Super Computer Center of Beijing Normal University at Zhuhai. This work was partially supported by the Natural Science Foundation of China (12271047); UIC research grant (R0400001-22; UICR0400008-21; UICR04202405-21); Guangdong College Enhancement and Innovation Program (2021ZDZX1046); Key Programme Special Fund in XJTLU (KSF-E-32), Research Enhancement Fund of XJTLU (REF-18-01-04); Guangdong Provincial Key Laboratory of Interdisciplinary Research and Application for Data Science, BNU-HKBU United International College (2022B1212010006).


 \bibliographystyle{elsarticle-num-names} 
 \bibliography{cas-refs}





\end{document}